\documentclass[a4paper,11pt,reqno]{amsart}

\usepackage{amsmath}
\usepackage{amsthm}
\usepackage{amssymb}
\usepackage{graphicx}
\usepackage{epsfig}
\usepackage{appendix}
\usepackage[pdftex]{color}
 \definecolor{MyRed}{rgb}{0.9,0,0}
 \definecolor{MyGreen}{rgb}{0,0.9,0}
 \definecolor{MyBlue}{rgb}{0,0,0.9}
\usepackage[pdftex]{hyperref}
\hypersetup{breaklinks=true,colorlinks=true,citecolor=MyGreen,urlcolor=MyRed,linkcolor=MyBlue,pdfpagemode=UseNone}
\theoremstyle{plain}

\newtheorem{theorem}{Theorem}[section]

\newtheorem{lemma}[theorem]{Lemma}
\newtheorem{corollary}[theorem]{Corollary}
\newtheorem{conjecture}[theorem]{Conjecture}

\renewcommand{\leq}{\leqslant}
\renewcommand{\geq}{\geqslant}

\def\COMMENT#1{}
\let\COMMENT=\footnote

\begin{document}

\title{A $q$-analogue of the four functions theorem}
\author{Demetres Christofides}
\email{christod@maths.bham.ac.uk}
\thanks{Supported by the EPSRC, grant no. EP/E02162X/1.}
\address{School of Mathematics, University of
Birmingham, Edgbaston, Birmingham, B15 2TT, UK}
\date{\today}
\begin{abstract}
In this article we give a proof of a $q$-analogue of the celebrated
four functions theorem. This analogue was conjectured by Bj\"orner
and includes as special cases both the four functions theorem and
also Bj\"orner's $q$-analogue of the FKG inequality.
\end{abstract}
\maketitle

\section{Introduction}

We denote the set of the first $n$ positive integers by $[n]$ and
the power set of $[n]$ by $\mathcal{P}(n)$. Given families
$\mathcal{A},\mathcal{B} \subseteq \mathcal{P}(n)$ we write
$\mathcal{A} \vee \mathcal{B} = \{A \cup B: A \in \mathcal{A},B \in
\mathcal{B}\}$ and $\mathcal{A} \wedge \mathcal{B} = \{A \cap B: A
\in \mathcal{A},B \in \mathcal{B}\}$. Given $f: \mathcal{P}(n) \to
\mathbb{R}$ and $\mathcal{A} \subseteq \mathcal{P}(n)$ we let
$f(\mathcal{A}) = \sum_{A \in \mathcal{A}}f(A)$. The four functions
theorem of Ahlswede and Daykin~\cite{Ahlswede&Daykin78} states the
following.

\begin{theorem}\label{4FT}
Let $\alpha,\beta,\gamma$ and $\delta$ be functions from
$\mathcal{P}(n)$ to the set of non-negative reals satisfying
\[
\alpha(A)\beta(B) \leq \gamma(A \cup B) \delta(A \cap B)
\]
for every $A,B \in \mathcal{P}(n)$. Then
\[
\alpha(\mathcal{A})\beta(\mathcal{B}) \leq \gamma(\mathcal{A} \vee
\mathcal{B}) \delta(\mathcal{A} \wedge \mathcal{B})
\]
for every $\mathcal{A},\mathcal{B} \subseteq \mathcal{P}(n)$.
\end{theorem}

This inequality generalized several well-known inequalities and has
found many applications. We refer the interested reader to the books
by Anderson~\cite{Anderson02} and Bollob\'as~\cite{Bollobas86} and
their relevant references for further discussions of this inequality
and its applications.

To state the next result we need to recall some facts from lattice
theory. We refer the reader to the classical book of
Birkhoff~\cite{Birkhoff79} for further information on lattices. A
\emph{lattice} $L$ is a partially ordered set in which every pair of
elements $x,y \in L$ has a unique least upper bound (denoted by $x
\vee y$) and a unique greatest lower bound (denoted by $x \wedge
y$). The lattice $L$ is called \emph{distributive} if for every
$x,y,z \in L$ the following holds:
\[
x \wedge (y \vee z) = (x \wedge y) \vee (x \wedge z).
\]
It is easy to check that $\mathcal{P}(n)$ is a distributive lattice.
Birkhoff's representation theorem asserts that every finite
distributive lattice is isomorphic to a sublattice of
$\mathcal{P}(n)$ for some $n$. In particular, Theorem~\ref{4FT} has
the following immediate consequence.

\begin{corollary}\label{4FT-Lattice}
Let $L$ be a finite distributive lattice and let
$\alpha,\beta,\gamma$ and $\delta$ be functions from $L$ to the set
of non-negative reals satisfying
\[
\alpha(x)\beta(y) \leq \gamma(x \vee y) \delta(x \wedge y)
\]
for every $x,y \in L$. Then
\[
\alpha(X)\beta(Y) \leq \gamma(X \vee Y) \delta(X \wedge Y)
\]
for every $X,Y \subseteq L$.
\end{corollary}

Indeed to prove the above corollary we just embed $L$ in
$\mathcal{P}(n)$ for some suitable $n$, extend $\alpha,\beta,\gamma$
and $\delta$ to be~0 outside $L$ and then apply Theorem~\ref{4FT}.

Let $L$ be a lattice. A function $f:L \to \mathbb{R}$ is called
\emph{increasing} if $f(x) \leq f(y)$ whenever $x \leq y$ and
\emph{decreasing} if $f(x) \leq f(y)$ whenever $x \geq y$. A
function $\mu$ from $L$ to the set of non-negative reals is said to
be \emph{log-supermodular} if it satisfies
\[
\mu(x)\mu(y) \leq \mu(x \vee y)\mu(x \wedge y)
\]
for every $x,y \in L$. It is usually convenient to think of $\mu$ as
a measure on $L$. We take this approach here and we thus define
\[
\int f d\mu := \sum_{x\in L}f(x)\mu(x).
\]
The following correlation inequality due to Fortuin, Kasteleyn and
Ginibre~\cite{Fortuin&Kasteleyn&Ginibre71} is known as the FKG
inequality.

\begin{corollary}\label{FKG}
Let $L$ be a finite distributive lattice, $\mu$ a log-supermodular
function on $L$ and $f,g$ functions from $L$ to the set of
non-negative reals which are either both increasing or both
decreasing. Then
\[
\int f d\mu \int g d\mu \leq \int 1 d\mu \int fg d\mu.
\]
\end{corollary}

This can be proved by applying Corollary~\ref{4FT-Lattice} with
$\alpha = f \mu , \beta = g \mu, \gamma = \mu$ and $\delta = f g
\mu$. The FKG inequality originally arose in the study of Ising
ferromagnets and the random cluster model. It later found many
application in extremal and probabilistic combinatorics. Recently,
Bj\"orner obtained a $q$-analogue of this inequality. Before stating
it we need to introduce some further notation.

Given a finite lattice $L$ and an element $x$ of $L$ the \emph{rank}
or \emph{height} of $x$ is the length of the longest chain having
$x$ as a maximal element and is denoted by $r(x)$. If $L$ is
distributive then the rank function satisfies the following
\emph{modular law}:
\[
r(x) + r(y) = r(x \vee y) + r(x \wedge y) \text{ for every } x,y \in
L.
\]
Given a finite distributive lattice $L$ with rank function $r$ and
functions $f,\mu:L \to \mathbb{R}$ we define the polynomial
\[
P_\mu(f;q) = \int f(x)q^{r(x)} d\mu.
\]
Finally, given polynomials $P(q),R(q) \in \mathbb{R}[q]$, we write
$P(q) \ll R(q)$ to denote that all coefficients of the polynomial
$R(q) - P(q)$ are non-negative reals.

Bj\"orner's $q$-analogue of the FKG-inequality~\cite{Bjorner} reads
as follows.

\begin{theorem}\label{q-FKG}
Let $L$ be a finite distributive lattice, $\mu$ a log-supermodular
function on $L$ and $f,g$ functions from $L$ to the set of
non-negative reals which are either both increasing or both
decreasing. Then
\[
P_\mu(f;q) P_\mu(g;q) \ll P_\mu(1;q) P_\mu(fg;q).
\]
\end{theorem}

The FKG inequality is obtained from the above result by putting
$q=1$. Several applications of Theorem~\ref{q-FKG} can be found
in~\cite{Bjorner}. It is natural to ask whether there the
corresponding $q$-analogue of the four functions theorem is true.
Answering a question of Bj\"orner~\cite{Bjorner} we prove the
following result.

\begin{theorem}\label{q-4FT-Lattice}
Let $L$ be a finite distributive lattice and let
$\alpha,\beta,\gamma$ and $\delta$ be functions from $L$ to the set
of non-negative reals satisfying
\[
\alpha(x)\beta(y) \leq \gamma(x \vee y) \delta(x \wedge y)
\]
for every $x,y \in L$. Then
\[
\sum_{x \in X} \alpha(x)q^{r(x)} \sum_{x \in Y} \beta(x)q^{r(x)} \ll
\sum_{x \in X \vee Y} \gamma(x)q^{r(x)} \sum_{x \in X \wedge Y}
\delta(x)q^{r(x)}.
\]
for every $X,Y \subseteq L$.
\end{theorem}

It is easy to see that Theorem~\ref{q-4FT-Lattice} includes both
Theorem~\ref{4FT} and Theorem~\ref{q-FKG} as special cases. We thus
obtain a new proof of Theorem~\ref{q-FKG}. Unfortunately we do not
obtain a new proof of Theorem~\ref{4FT} as this theorem itself will
be used in the proof of Theorem~\ref{q-4FT-Lattice}.

We give the proof of Theorem~\ref{q-4FT-Lattice} in the next
section. In Section~3 we discuss a stronger conjecture which turns
out to be false. In the proof of Theorem~\ref{q-4FT-Lattice} we will
need to use a result which although not a corollary of Birkhoff's
representation theorem, it is a simple consequence of its proof. We
thus add an appendix with a proof of this result.

\section{Proof of Theorem~\ref{q-4FT-Lattice}}

We claim that it is enough to prove Theorem~\ref{q-4FT} in the case
when $L$ is the Boolean lattice $\mathcal{P}(n)$ and
$X=Y=\mathcal{P}(n)$. I.e.~it is enough to prove the following
theorem.

\begin{theorem}\label{q-4FT}
Let $\alpha,\beta,\gamma$ and $\delta$ be functions from
$\mathcal{P}(n)$ to the set of non-negative reals satisfying
\[
\alpha(A)\beta(B) \leq \gamma(A \cup B) \delta(A \cap B)
\]
for every $A,B \in \mathcal{P}(n)$. Then
\[
\sum_{A \in \mathcal{P}(n)} \alpha(A)q^{|A|} \sum_{B \in
\mathcal{P}(n)} \beta(B)q^{|B|} \ll \sum_{C \in \mathcal{P}(n)}
\gamma(C)q^{|C|} \sum_{D \in \mathcal{P}(n)} \delta(D)q^{|D|}.
\]
\end{theorem}

We begin by showing that it is indeed enough to prove the above
theorem.

\begin{proof}[Theorem~\ref{q-4FT} implies Theorem~\ref{q-4FT-Lattice}]
Let $\phi : L \to \mathcal{P}(n)$ be the embedding given by
Theorem~\ref{BRT}. For every $A \in \mathcal{P}(n)$ we let
$\alpha'(A)$ to be equal to $\alpha(x)$ if $x \in X$ and
$\phi(x)=A$. Otherwise we put $\alpha'(A)=0$. We define
$\beta',\gamma'$ and $\delta'$ analogously. Observe that the
functions $\alpha',\beta',\gamma'$ and $\delta'$ satisfy the
conditions of Theorem~\ref{q-4FT}. Indeed, $\alpha'(A)\beta'(B)$ is
non-zero only when $\phi(x)=A$ and $\phi(y) = B$ for some $x \in X,y
\in Y$. But in this case we have $A \cup B = \phi(x) \cup \phi(y) =
\phi(x \vee y)$ and thus $\gamma'(A \cup B) = \gamma(x \vee y)$.
Similarly we also have $\delta'(A \cap B') = \gamma(x \wedge y)$ and
thus $\alpha'(A)\beta'(B) \leq \gamma'(A \cup B) \delta'(A \cap B)$
as required. The inequality now follows as $r(a) = |\phi(a)|$ for
every $a \in L$.
\end{proof}

In fact, instead of Theorem~\ref{q-4FT} we will prove the following
stronger assertion.

\begin{theorem}\label{q-4FT-stronger}
Let $\alpha,\beta,\gamma$ and $\delta$ be functions from
$\mathcal{P}(n)$ to the set of non-negative reals satisfying
\[
\alpha(A)\beta(B) \leq \gamma(A \cup B) \delta(A \cap B)
\]
for every $A,B \in \mathcal{P}(n)$. Then
\[
\sum_{A \in \mathcal{P}(n)}\alpha(A)\beta(A^c) \leq \sum_{C \in
\mathcal{P}(n)}\gamma(C) \delta(C^c).
\]
\end{theorem}

Before proving Theorem~\ref{q-4FT-stronger} we show how it implies
Theorem~\ref{q-4FT}.

\begin{proof}[Theorem~\ref{q-4FT-stronger} implies Theorem~\ref{q-4FT}]
We need to prove one inequality for each power of $q$. For the
coefficients of $q^k$ we need to prove that
\[
\sum_{|A| + |B| = k} \alpha(A) \beta(B) \leq \sum_{|C| + |D| = k }
\gamma(C)\delta(D).
\]
Since
\[
\sum_{|A| + |B| = k} \alpha(A) \beta(B) = \sum_{F,G}
\mathop{\sum_{|A| + |B|=k}}_{A \cap B = F,A \cup B = G} \alpha(A)
\beta(B),
\]
it is enough to prove that for each $F,G \in \mathcal{P}(n)$ the
following inequality holds:
\[
\mathop{\sum_{|A|+|B|=k}}_{A \cap B = F,A \cup B = G} \alpha(A)
\beta(B) \leq \mathop{\sum_{|C|+|D|=k}}_{C \cap D = F, C \cup D = G}
\gamma(C) \delta(D).
\]
The inequality is vacuously true unless $F \subseteq G$ and $|F| +
|G| = 2k$. In this case, we apply Theorem~\ref{q-4FT-stronger} on
the functions $\alpha',\beta',\gamma'$ and $\delta'$ from
$\mathcal{P}(G \setminus F)$ to $\mathbb{R}$, where $\alpha'(A) =
\alpha(A \cup F)$ and $\beta',\gamma'$ and $\delta'$ are defined
analogously. Since
\[
\alpha'(A)\beta'(B) = \alpha(A \cup F)\beta(B \cup F) \leq \gamma(A
\cup B \cup F) \delta((A \cap B) \cup F) = \gamma'(A \cup
B)\delta'(A \cap B),
\]
Theorem~\ref{q-4FT-stronger} gives
\[
\sum_{A \in \mathcal{P}(G \setminus F)}\alpha'(A)\beta'(A^c) \leq
\sum_{C \in \mathcal{P}(G \setminus F)}\gamma'(C) \delta'(C^c).
\]
or equivalently
\[
\sum_{A \in \mathcal{P}(G \setminus F)}\alpha(A \cup F)\beta(G
\setminus A) \leq \sum_{C \in \mathcal{P}(G \setminus F)}\gamma(C
\cup F) \delta(G \setminus C ).
\]
But this is exactly the inequality we wanted to prove.
\end{proof}

We need one more lemma before proceeding to the proof of
Theorem~\ref{q-4FT-stronger}.

\begin{lemma}\label{4FT-setminus version}
Let $\alpha,\beta,\gamma$ and $\delta$ be functions from
$\mathcal{P}(n)$ to the set of non-negative reals satisfying
\[
\alpha(A)\beta(B) \leq \gamma(B \setminus A) \delta(A \setminus B)
\]
for every $A,B \in \mathcal{P}(n)$. Then
\[
\alpha(\mathcal{P}(n)) \beta(\mathcal{P}(n)) \leq
\gamma(\mathcal{P}(n)) \delta(\mathcal{P}(n)).
\]
\end{lemma}

\begin{proof}
We just apply Theorem~\ref{4FT} to the functions
$\alpha,\beta',\gamma'$ and $\delta$, where $\beta'(B) :=
\beta(B^c)$ and $\gamma'(C) := \gamma(C^c)$. The lemma follows
directly since for any $A,B \in \mathcal{P}(n)$ we have
\[
\alpha(A)\beta'(B) = \alpha(A) \beta(B^c) \leq \gamma(A \cup B^c)
\delta (A \cap B^c) = \gamma'(B \setminus A) \delta(A \setminus B).
\qedhere
\]
\end{proof}

We are now ready to prove Theorem~\ref{q-4FT-stronger}. This will
complete the proof of Theorem~\ref{q-4FT-Lattice}.

\begin{proof}[Proof of Theorem~\ref{q-4FT-stronger}]
Let us define $f:\mathcal{P}(n) \to \mathbb{R}$ by $f(A) = \alpha(A)
\beta(A^c)$ and $g:\mathcal{P}(n) \to \mathbb{R}$ by $g(A) =
\gamma(A^c) \delta(A)$. Then
\begin{align*}
f(A)f(B) &= \left(\alpha(A) \beta(B^c) \right) \left(\alpha(A^c) \beta(B) \right)\\
&\leq \gamma(A \cup B^c) \delta(A \cap B^c) \gamma(A^c \cup B)
\delta(A^c \cap B) \\
&= g(A \cap B^c) g(A^c \cap B) = g(B \setminus A)g(A \setminus B).
\end{align*}
Thus, applying Lemma~\ref{4FT-setminus version} with $\alpha = \beta
= f$ and $\gamma=\delta=g$, we obtain that
\[
\left(f(\mathcal{P}(n))\right)^2 \leq
\left(g(\mathcal{P}(n))\right)^2.
\]
The result now follows as all functions used take only non-negative
values.
\end{proof}

\section{A counterexample to a stronger conjecture}

In our attempt to prove Theorem~\ref{q-4FT-Lattice} we arrived at
the following conjecture.

\begin{conjecture}\label{False conjecture}
Let $\alpha,\beta,\gamma$ and $\delta$ be functions from
$\mathcal{P}(n)$ to the set of non-negative reals satisfying
\[
\alpha(A)\beta(B) \leq \gamma(A \cup B) \delta(A \cap B)
\]
for every $A,B \in \mathcal{P}(n)$. Then
\[
\alpha(A) \beta(B) + \alpha(B) \beta(A) \leq \gamma(A) \delta(B) +
\gamma(B) \delta(A).
\]
\end{conjecture}

Conjecture~\ref{False conjecture} is easily seen to imply
Theorem~\ref{q-4FT-stronger} and thus Theorem~\ref{q-4FT}. (In fact,
it is easy to deduce Theorem~\ref{q-4FT} from Conjecture~\ref{False
conjecture} directly without going through
Theorem~\ref{q-4FT-stronger}.) It can be checked that the conjecture
is true for $n=1$ and a simple inductive argument which we omit
shows that if the conjecture were true in the case $n=2$ as well,
then it would be true for every positive integer $n$. Unfortunately,
it turns out that the conjecture is false in the case $n=2$ as can
be checked by defining $\alpha,\beta,\gamma$ and $\delta$ as
suggested in the following table.\vspace{10pt}

\begin{center}
\begin{tabular}{c|c|c|c|c}
            & $\alpha$ & $\beta$ & $\gamma$ & $\delta$ \\ \hline
$\emptyset$ & 0        & 1       & 0        & 1 \\ \hline
$\{1\}    $ & 0        & 1       & 0        & 0 \\ \hline
$\{2\}    $ & 1        & 1       & 1        & 1 \\ \hline
$\{1,2\}  $ & 0        & 0       & 1        & 0 \\
\end{tabular}
\end{center}\vspace{10pt}

There are only three pairs $(A,B)$ for which $\alpha(A) \beta(B)$ is
non-zero and in each one of them one can check that the inequality
$\alpha(A)\beta(B) \leq \gamma(A \cup B) \delta(A \cap B)$ holds.
However, taking $A = \{1\}$ and $B = \{2\}$ we have $\alpha(A)
\beta(B) + \alpha(B) \beta(A) = 1 > 0 = \gamma(A) \delta(B) +
\gamma(B) \delta(A)$ thus disproving the conjecture.

\appendix
\section{}

In this appendix we give a proof of Birkhoff's representation
theorem. It is usually stated without condition (iii) but this
condition, which is needed in the proof of
Theorem~\ref{4FT-Lattice}, follows easily from the usual proof of
the theorem.

\begin{theorem}\label{BRT}
Let $L$ be a finite distributive lattice. Then there is a positive
integer $n$ and an injective function $\phi:L \to \mathcal{P}[n]$
satisfying
\begin{itemize}
\item[(i)] $\phi(a \wedge b) = \phi(a) \cap \phi(b)$ for all $a,b \in
L$;
\item[(ii)] $\phi(a \vee b) = \phi(a) \cup \phi(b)$ for all $a,b \in
L$;
\item[(iii)] $r(a) = |\phi(a)|$ for all $a \in L$.
\end{itemize}
\end{theorem}

\begin{proof}
We say that a non-zero element $x$ of $L$ is join-irreducible if $a
\vee b < x$ whenever $a < x$ and $b < x$. (Recall that every finite
lattice $L$ has a unique minimal element which is called the zero of
$L$.) Let $X = \{x_1,\ldots,x_n\}$ be the set of join-irreducible
elements of $L$. By reordering the elements of $X$, we may assume
that if $x_i < x_j$ then $i < j$. For each $a \in L$ define $\phi(a)
= \{i: x_i \leq a\}$. To see that $\phi$ is injective observe that
if $a \nleqslant b$ then the set $S = \{x \in L: x \leq a,x
\nleqslant b\}$ is non-empty (as $b \in S$) and contains a
join-irreducible element. Indeed if $x$ is a minimal element of $S$
and there are $c,d \in L$ with $c,d < x$ and $c \vee d = x$, then by
minimality of $x$ we have that $c,d \notin S$ and so it must be the
case that $c,d \leqslant b$. But then $x = c \vee d \leqslant b$, a
contradiction. Observe that (i) follows immediately from the
definition of $\phi$. It also follows immediately from the
definition that $\phi(a) \cup \phi(b) \subseteq \phi(a \vee b)$. To
complete the proof of (ii) observe that if $x \in X$ with $x \leq a
\vee b$ but $x \nleqslant a$ and $x \nleqslant b$, then $x \wedge a
< x, x \wedge b < x$ but $(x \wedge a) \vee (x \wedge b) = x \wedge
(a \vee b) = x$, contradicting the fact that $x \in X$. It remains
to prove (iii). Let $a \in L$ and suppose that $\phi(a) =
\{i_1,\ldots,i_k\}$, where $i_1 < \cdots < i_k$. Since $0 < x_{i_1}
< x_{i_1} \vee x_{i_2} < \cdots < x_{i_1} \vee \cdots \vee x_{i_k}
\leq a$, it follows that $|\phi(a)| = k \leq r(a)$. On the other
hand, if $0 < y_1 < y_2 < \ldots < y_r = a$ is a longest chain
having $a$ as a maximal element, then, since $\phi$ is injective, we
have $\emptyset \varsubsetneq \phi(y_1) \varsubsetneq \phi(y_2)
\varsubsetneq \cdots \varsubsetneq \phi(y_r)$ showing that
$|\phi(a)| \geq r = r(a)$. This completes the proof of (iii) and
thus of the theorem.
\end{proof}

\end{document}